\documentclass[oneside,11pt]{amsart}
\usepackage{amsmath, amsfonts, amsthm, times, graphics}

\usepackage[active]{srcltx}
 \makeatletter
\renewcommand*\subjclass[2][2000]{%
  \def\@subjclass{#2}%
  \@ifundefined{subjclassname@#1}{%
    \ClassWarning{\@classname}{Unknown edition (#1) of Mathematics
      Subject Classification; using '1991'.}%
  }{%
    \@xp\let\@xp\subjclassname\csname subjclassname@#1\endcsname
  }%
}
 \makeatother
\newtheorem{theorem}{Theorem}[section]
\newtheorem{lemma}[theorem]{Lemma}
\newtheorem{corollary}[theorem]{Corollary}
\newtheorem{proposition}[theorem]{Proposition}
\theoremstyle{definition}

\newtheorem{remark}[theorem]{Remark}
\numberwithin{equation}{section}

 \makeatletter
\renewcommand*\subjclass[2][2000]{%
  \def\@subjclass{#2}%
  \@ifundefined{subjclassname@#1}{%
    \ClassWarning{\@classname}{Unknown edition (#1) of Mathematics
      Subject Classification; using '1991'.}%
  }{%
    \@xp\let\@xp\subjclassname\csname subjclassname@#1\endcsname
  }%
}

\def\NABLA#1{{\mathop{\nabla\kern-.5ex\lower1ex\hbox{$#1$}}}}
\def\Nabla#1{\nabla\kern-.5ex{}_{#1}}
\def\Tabla#1{\Tilde\nabla\kern-.5ex{}_{#1}}
\renewcommand{\Tilde}{\widetilde}

\makeatother

\begin{document}

\title {A Sharp inequality for holomorphic functions on the polydisc}

\author{Marijan Markovi\'c}
\address{University of Montenegro, Faculty of Natural Sciences and Mathematics,
Cetinjski put b.b. 81000, Podgorica, Montenegro}
\email{marijanmmarkovic@gmail.com}



\begin{abstract} In this paper we prove an isoperimetric inequality for holomorphic
functions in the unit polydisc $\mathbf U^n$. As a corollary we derive an inclusion
relation between weighted Bergman and Hardy spaces of holomorphic functions in the
polydisc which generalizes the classical Hardy--Littlewood relation $H^p\subseteq A^{2p}$.
Also, we extend some results due to Burbea.
\end{abstract}

\maketitle

\section{Introduction and statement of the result}

\subsection{Notations}
For an integer $n\ge1$ we consider the $n$-dimensional complex vector space $\mathbf C^n$
with the usual inner product
$$\left<z,\zeta\right>=\sum_{j=1}^nz_j\overline{\zeta}_j,\quad z,\ \zeta \in \mathbf C^n$$
and norm
$$\|z\|=\left<z,z\right>^\frac 12,$$
where $z=(z_1,\dots,z_n)$, $\zeta=(\zeta_1,\dots,\zeta_n)$, $\mathbf U$ is the
open unit disc in the complex plane $\mathbf C$, $\mathbf T$ is its boundary,
and $\mathbf U^n$ and $\mathbf T^n$ stand for the unit polydisc and its distinguished
boundary, respectively.

Following the classical book of Rudin~\cite{rudin}, let us recall some basic facts
from the theory of Hardy spaces $H^p(\mathbf U^n)$ on the unit polydisc. Let $p>0$
be an arbitrary real (in the sequel the letter $p$, with or without index, will stand
for a positive real number, if we do not make restrictions). By $dm_n$ we denote the
Haar measure on the distinguished boundary $\mathbf T^n$, i.e.,
$$dm_n(\omega)=
\frac 1{(2\pi)^n}d\theta_1 \dots d\theta_n,
\quad \omega=(e^{i\theta_1},\dots,e^{i\theta_n}) \in \mathbf T^n.$$
A holomorphic function $f$ in the polydisc $\mathbf U^n$ belongs to the Hardy space
$H^p(\mathbf U^n)$ if it satisfies the growth condition
\begin{equation}\label{growthhardy}
\|f\|_{H^p(\mathbf U^n)}:=
\left(\sup_{0\le r<1}\int_{\mathbf T^n}|f(r\omega)|^p dm_n(\omega)\right)^\frac 1p
<\infty.
\end{equation}
It turns out that if $f \in H^p(\mathbf U^n)$, then there exists
$$\lim_{r\to 1}f(r\omega)=f(\omega)\ \text{ a.e. on }\ \mathbf T^n$$
and the boundary function belongs to $L^p(\mathbf T^n, m_n)$, the Lebesgue space of
all $p$-integrable functions on $\mathbf T^n$ (with respect to the measure $m_n$).
Moreover
$$\int_{\mathbf T^n}|f(\omega)|^p dm_n(\omega)
=\sup_{0\le r<1}\int_{\mathbf T^n}|f(r\omega)|^p dm_n(\omega).$$

For $q>-1$ let
$$d\mu_q(z)=\frac{(q+1)}\pi(1-|z|^2)^q dxdy \quad (z=x+iy \in \mathbf U),$$
stand for a weighted normalized measure on the disc $\mathbf U$. We will consider
also the corresponding measure on the polydisc $\mathbf U^n$,
$$d\mu_\mathbf q(z)=\prod_{k=1}^n d\mu_{q_k}(z_k),\quad z \in \mathbf U^n,$$
where $\mathbf q>\mathbf {-1}$ is an $n$-multiindex; the inequality $\mathbf {q}_1>
\mathbf{q}_2$ between two $n$-multiindices means $q_{1,k}>q_{2,k},\ k=1,\dots,n$; we
denote the $n$-multiindex $(m,\dots,m)$ by $\mathbf m$. For a real number $m,\ m>1$,
we have
$$d\mu_\mathbf {m-2}(z)=
\frac{(m-1)^n}{\pi^n}\prod_{k=1}^n(1-|z_k|^2)^{m-2}dx_kdy_k \quad (z_k=x_k+iy_k).$$

The weighted Bergman spaces $A^p_\mathbf q(\mathbf U^n),\ p>0,\ \mathbf q>\mathbf {-1}$
contain the holomorphic functions $f$ in the polydisc $\mathbf U^n$ such that
$$\|f\|_{A^p_\mathbf q (\mathbf U^n)}:=
\left(\int_{\mathbf U^n}|f(z)|^p d\mu_\mathbf q(z)\right)^\frac 1p <\infty.$$
Since $d\mu_0$ is the area measure on the complex plane normalized on the unit disc,
$A^p(\mathbf U^n):=A^p_\mathbf 0(\mathbf U^n)$ are the ordinary (unweighted) Bergman
spaces on $\mathbf U^n$.

It is well known that $\|\cdot\|_{H^p(\mathbf U^n)}$ and $\|\cdot\|_{A^p_\mathbf q(\mathbf U^n)}$
are norms on $H^p(\mathbf U^n)$ and $A^p_\mathbf q(\mathbf U^n)$, respectively, if
$p\ge1$, and quasi-norms for $0<p<1$; for simplicity, we sometimes write
$\|\cdot\|_p$ and $\|\cdot\|_{p,\mathbf q}$. As usual, $H^p(\mathbf U)$ and
$A^p_q(\mathbf U)$ are denoted by $H^p$ and $A^p_q$. Obviously, $H^p(\mathbf U^n)
 \subseteq A^p_\mathbf q(\mathbf U^n)$.

Let us point out that the Hardy space $H^2(\mathbf U^n)$ is a Hilbert space with
the reproducing kernel
\begin{equation}\label{repker}
K_n(z,\zeta)=
\prod_{j=1}^n\frac 1{1-z_j \overline{\zeta}_j},\quad z,\ \zeta \in \mathbf U^n.
\end{equation}
For the theory of reproducing kernels we refer to~\cite{aron}.
\subsection{Short background} The solution to the isoperimetric problem is usually
expressed in the form of an inequality that relates the length $L$ of a closed
curve and the area $A$ of the planar region that it encloses. The isoperimetric
inequality states that
$$4\pi A \le L^2,$$
and that equality holds if and only if the curve is a circle. Dozens of proofs
of the isoperimetric inequality have been proposed. More than one approach can be
found in the expository papers by Osserman~\cite{or}, Gamelin and Khavinson~\cite{gk}
and Bl\"asj\"o~\cite{bla} along with a brief history of the problem. For a survey
of some known generalizations to higher dimensions and the list of some open
problems, we refer to the paper by B$\mathrm{\acute{e}}$n$\mathrm{\acute{e}}$teau
and Khavinson~\cite{bek}.

In~\cite{tc}, Carleman gave a beautiful proof of the isoperimetric inequality in
the plane, reducing it to an inequality for holomorphic functions in the unit disc.
Following Carleman's result, Aronszajn in~\cite{aron} showed that if $f_1$ and
$f_2$ are holomorphic functions in a simply connected domain $\Omega$ with analytic
boundary $\partial \Omega$, such that $f_1,\ f_2\in H^2(\Omega)$, then
\begin{equation}\label{car}
\int_\Omega |f_1|^2 |f_2|^2 dxdy
\le \frac 1{4\pi}\int_{\partial \Omega}|f_1|^2|dz|\int_{\partial \Omega}|f_2|^2|dz|
\quad (z=x+iy).
\end{equation}
In~\cite{jacobs} Jacobs considered not only simply connected domains (see also the
Saitoh work~\cite{ss}).

Mateljevi\'c and Pavlovi\'c in~\cite{mp} generalized~\eqref{car} in the following
sense: if $f_j\in H^{p_j}(\Omega),\ j=1,\ 2$, where $\Omega$ is a simply connected
domain with analytic boundary $\partial \Omega$, then
\begin{equation}
\frac 1\pi\int_\Omega |f_1|^{p_1}|f_2|^{p_2}dxdy
\le
\frac 1{4\pi^2}\int_{\partial \Omega}|f_1|^{p_1}|dz| \int_{\partial \Omega}|f_2|^{p_2}|dz|,
\end{equation}
with equality if and only if either $f_1f_2\equiv0$ or if for some $C_j\ne0$,
$$f_j= C_j(\psi')^\frac 1{p_j},\quad j=1,\ 2,$$
where $\psi$ is a conformal mapping of the domain $\Omega$ onto the disc $\mathbf U$.

By using a similar approach as Carleman, Strebel in his book~\cite[Theorem~19.9, pp. 96--98]{ks}
(see also the papers~\cite{malect} and~\cite{vuk}) proved that if $f\in H^p$ then
\begin{equation}\label{strebel}
\int_\mathbf U|f(z)|^{2p}dxdy
\le\frac 1{4\pi} \left(\int_0^{2\pi}|f(e^{i\theta})|^pd\theta\right)^2
\quad (\|f\|_{A^{2p}} \le \|f\|_{H^p}),
\end{equation}
with equality if and only if for some constants $\zeta,\ |\zeta|<1$ and $\lambda$,
$$f(z)=\frac{\lambda}{(1-z\overline {\zeta})^\frac 2p}.$$
Further, Burbea in~\cite{burbea} generalized~\eqref{strebel} to
\begin{equation}\label{bur}
\frac{m-1}\pi\int_\mathbf U |f(z)|^{mp}(1-|z|^2)^{m-2} dxdy
\le \left(\frac 1{2\pi}\int_0^{2\pi}|f(e^{i\theta})|^p d\theta\right)^m,
\end{equation}
where $m\ge2$ is an integer. The equality is attained in the same case as in the
relation~\eqref{strebel}. The inequality~\eqref{bur} can be rewritten as
$$\|f\|_{A^{mp}_{m-1}} \le \|f\|_{H^p}, \quad f\in H^p,$$
which is a generalization of the inclusion $ H^p\subseteq A^{2p}$, proved by Hardy
and Littlewood in~\cite{hardy}.

On the other hand, Pavlovi\'c and Dostani\'c showed in~\cite{pd} that if $\mathbf B_n$
is the unit ball in $\mathbf C^n,\ \partial\mathbf B_n$ its boundary, and $\sigma_n$
is the normalized surface area measure on the sphere $\partial\mathbf B_n$, then
$$\int_{\partial\mathbf B_n}|f|^{2n}d\sigma_n \le \left(\int_{\mathbf T^n}|f|^2dm_n\right)^n$$
holds for $f\in H^2(\mathbf U^n)$. They pointed out that this inequality coincides
with~\eqref{bur} for $m=n,\ p=2$ and $f(z)=f(z_1,\dots,z_n)=f(z_1)$ that is if $f$
actually depends only on one complex variable.

For an isoperimetric inequality for harmonic functions we refer to~\cite{kame}.
\subsection{Statement of the result} In the sequel, $m$ stands for an integer $\ge2$.
The starting point of this paper is the work of Burbea~\cite{burbea} who obtained the
following isoperimetric inequalities concerning the unit disc and the unit polydisc.

\begin{proposition}\label{burbeadisc} Let $f_j \in H^{p_j},\ j=1,\dots,m$. Then
\begin{equation}\label{burbea}
\int_\mathbf U \prod_{j=1}^m|f_j|^{p_j} d\mu_{m-2}
\le \prod_{j=1}^m\int_\mathbf T|f_j|^{p_j} dm_1.
\end{equation}

Equality holds if and only if either some of the functions are identically equal to
zero or if for some point $\zeta \in \mathbf U$ and constants $C_j\ne 0$,
$$f_j=C_j K_1^\frac 2{p_j}(\cdot,\zeta),\quad j=1,\dots,m,$$
where $K_1$ is the reproducing kernel~\eqref{repker} for the Hardy space $H^2$.
\end{proposition}

\begin{proposition}\label{burbeapoly} Let $f_j \in H^2(\mathbf U^n),\ j=1,\dots,m$.
Then
\begin{equation}
\int_{\mathbf U^n}\prod_{j=1}^m|f_j|^2 d\mu_\mathbf {m-2}
\le\prod_{j=1}^m\int_{\mathbf T^n}|f_j|^2 dm_n.
\end{equation}

Equality holds if and only if either some of the functions are identically equal to
zero or if for some point $\zeta \in \mathbf U^n$ and constants $C_j\ne0$,
$$f_j=C_j K_n(\cdot,\zeta),\quad j=1,\dots, m,$$
where $K_n$ is the reproducing kernel~\eqref{repker}.
\end{proposition}

Proposition~\ref{burbeapoly} is a particular case of Theorem~4.1 in the Burbea
paper~\cite[p. 257]{burbea}. That theorem was derived from more general
considerations involving the theory of reproducing kernels (see also~\cite{bur0}).
The inequality in that theorem is between Bergman type norms, while Proposition~\ref{burbeapoly}
is the case with the Hardy norm on the right side (in that case, we have an
isoperimetric inequality). In the next main theorem we extend~\eqref{burbea} for
holomorphic functions which belong to general Hardy spaces on the polydisc $\mathbf U^n$.

\begin{theorem}\label{mainth} Let $f_j\in H^{p_j}(\mathbf U^n),\ j=1,\dots,m$.
Then
\begin{equation}\label{maineq}
\int_{\mathbf U^n}\prod_{j=1}^m|f_j|^{p_j}d\mu_\mathbf{m-2}
\le \prod_{j=1}^m\int_{\mathbf T^n}|f_j|^{p_j} dm_n.
\end{equation}

Equality occurs if and only if either some of the functions are identically equal to
zero or if for some point $\zeta \in \mathbf U^n$ and constants $C_j\ne0$,
$$f_j=C_jK_n^\frac 2{p_j}(\cdot,\zeta),\quad j=1,\dots, m.$$
\end{theorem}

Notice that in higher complex dimensions there is no analog of the Blaschke product so
we cannot deduce Theorem~\ref{mainth} directly from Proposition~\ref{burbeapoly} as we
can for $n=1$ (this is a usual approach in the theory of $H^p$ spaces; see also~\cite{burbea}).
We will prove the main theorem in the case $n=2$ since for $n>2$ our method needs only
a technical adaptation. As immediate consequences of Theorem~\ref{mainth}, we have the
next two corollaries.

\begin{corollary} Let $p\ge1$. The (polylinear) operator
$\beta:\bigotimes_{j=1}^m H^p(\mathbf U^n) \to A^{p}_\mathbf{m-2}(\mathbf U^n)$,
defined by $\beta(f_1,\dots,f_m)=\prod_{j=1}^mf_j$ has norm one.
\end{corollary}

\begin{corollary} Let $f\in H^p(\mathbf U^n)$. Then
$$\int_{\mathbf U^n}|f|^{mp}d\mu_\mathbf {m-2} \le \left(\int_{\mathbf T^n} |f|^p dm_n\right)^m.$$
Equality occurs if and only if for some point $\zeta \in \mathbf U^n$ and constant
$\lambda$,
$$f=\lambda K_n^\frac 2{p}(\cdot,\zeta).$$

In other words we have the sharp inequality
$$\|f\|_{A^{mp}_\mathbf {m-2} (\mathbf U^n)} \le \|f\|_{H^p(\mathbf U^n)},
\quad f \in H^{p}(\mathbf U^n)$$
and the inclusion
$$H^p(\mathbf U^n) \subseteq A^{mp}_\mathbf {m-2}(\mathbf U^n).$$
Thus, when $p\ge1$, the inclusion map
$I_{p,m}: H^p(\mathbf U^n) \to A^{mp}_\mathbf {m-2}(\mathbf U^n),\ I_{p,m}(f):=f$,
has norm one.
\end{corollary}

\section{Proof of the main theorem}

Following Beckenbach and Rad$\acute{\mathrm{o}}$~\cite{ber}, we say that a non-negative
function $u$ is logarithmically subharmonic in a plane domain $\Omega$ if $u\equiv0$ or
if $\log u$ is subharmonic in $\Omega$.

Our first step is to extend the Burbea inequality~\eqref{burbea} to functions which
belong to the spaces $h^p_{PL}$ defined in the following sense: $u \in h^p_{PL}$
if it is logarithmically subharmonic and satisfies the growth property
\begin{equation}\label{growth}
\sup_{0\le r<1}\int_0^{2\pi}|u(re^{i\theta})|^p\frac{d\theta}{2\pi}<\infty.
\end{equation}
It is known that a function from these spaces has a radial limit in $e^{i\theta} \in
 \mathbf T$ for almost all $\theta \in [0,2\pi]$. Let us denote this limit (when it
exists) as $u(e^{i\theta})$. Then $\lim_{r\to 1}\int_0^{2\pi}|u(re^{i\theta})|^p\frac{d\theta}{2\pi}=
\int_0^{2\pi}|u(e^{i\theta})|^p\frac{d\theta}{2\pi}$ and
$\lim_{r\to 1}\int_0^{2\pi}|u(re^{i\theta})-u(e^{i\theta})|^p\frac{d\theta}{2\pi}=0$.
For an exposition of the topic of spaces of logarithmically subharmonic functions
which satisfy~\eqref{growth}, we refer to the book of Privalov~\cite{privalov}.

\begin{lemma}\label{loglema} Let $u_j\in h^{p_j}_{PL},\ j=1,\dots,m$, be logarithmically
subharmonic functions in the unit disc. Then
\begin{equation} \label{burbealog}
\int_\mathbf U \prod_{j=1}^m u_j^{p_j} d\mu_{m-2} \le \prod_{j=1}^m \int_{\mathbf T}u_j^{p_j} dm_1.
\end{equation}

For continuous functions, equality holds if and only if either some of the functions
are identically equal to zero or if for some point $\zeta \in \mathbf U$ and constants
$\lambda_j>0$,
$$u_j=\lambda_j\left|K_1(\cdot,\zeta)\right|^\frac 2{p_j},\quad j=1,\dots,m.$$
\end{lemma}

\begin{proof} Suppose that no one of the functions $u_j,\ j=1,\dots,m$ is identically
equal to zero. Then $\log u_j(e^{i\theta})$ is integrable on the segment $[0,2\pi]$ and
there exist $f_j \in H^{p_j}$ such that $u_j(z)\le |f_j(z)|,\ z \in \mathbf U$ and
$u_j(e^{i\theta})=|f_j(e^{i\theta})|,\ \theta \in [0,2\pi]$. Namely, for $f_j$ we can
choose
$$f_j(z)=
\exp\left(\int_0^{2\pi}\frac{e^{i\theta}+z}{e^{i\theta}-z}\log u_j(e^{i\theta}) \frac{d\theta}{2\pi}\right),
\quad j=1,\dots,m.$$
Since $\log u_j$ is subharmonic we have
$$\log u_j(z)
\le \int_0^{2\pi}P(z,e^{i\theta})\log u_j(e^{i\theta})\frac{d\theta}{2\pi},$$
where $P(z,e^{i\theta})=\frac{1-|z|^2}{|z-e^{i\theta}|^2}$ is the Poisson kernel for
the disc $\mathbf U$. From this it follows the $u_j(z) \le |f_j(z)|,\ z \in \mathbf U$.
Moreover, using the Jensen inequality (for the concave function $\log$) we obtain
\[\begin{split}
\log |f_j(z)|^{p_j} &=
\int_0^{2\pi}P(z,e^{i\theta})\log|u_j(e^{i\theta})|^{p_j}\frac{d\theta}{2\pi}
\\&\le\log\int_0^{2\pi}P(z,e^{i\theta})u^{p_j}_j(e^{i\theta})\frac{d\theta}{2\pi},
\end{split}\]
implying $f_j\in H^{p_j}$.

Now, in~\eqref{burbea} we can take $f_j,\ j=1,\dots, m$, and use the previous
relations, $u_j(z)\le |f_j(z)|,\ z \in \mathbf U$ and $u_j(e^{i\theta})=|f_j(e^{i\theta})|,
\ \theta \in [0,2\pi]$, to derive the inequality~\eqref{burbealog}.

If all of the functions $u_j,\ j=1,\dots,m$ are continuous (not equal to zero
identically) and if the equality in~\eqref{burbealog} occurs, then $u_j(z)=|f_j(z)|,\ z
 \in \mathbf U, \ j=1,\dots,m$. According to the equality in the Proposition~\ref{burbeadisc},
we must have $f_j=C_j K_1^\frac 2{p_j}(\cdot,\zeta),\ j=1,\dots, m$, for some point
$\zeta \in \mathbf U$ and constants $C_j\ne 0$. Thus, for continuous functions (not
equal to zero identically) equality holds if and only if
$u_j=\lambda_j|K_1(\cdot,\zeta)|^\frac 2{p_j},\ j=1,\dots,m\ (\lambda_j>0)$.
\end{proof}

We need the next two propositions concerning (logarithmically) subharmonic functions.
For the proofs of these propositions see the first paragraph of the book of Ronkin~\cite{ronkin}.

\begin{proposition}\label{ron1} Let $f$ be an upper semi-continuous function on a
product $\Omega \times \Delta$ of domains $\Omega \subseteq \mathbf R^n$ and
$\Delta \subseteq \mathbf R^k$. Let $\mu$ be a positive measure on $\Delta$ and
$E \subseteq \Delta$ be such that $\mu(E)<\infty$. Then
$$\varphi(x):=\int_E f(x,y)d\mu(y),\quad x \in \Omega$$
is (logarithmically) subharmonic if $f(\cdot,y)$ is (logarithmically) subharmonic
for all (almost all with respect to the measure $\mu$) $y \in \Omega$.
\end{proposition}

\begin{proposition}\label{ron2} Let $A$ be an index set and $\{u_{\alpha},\ \alpha
 \in A\}$ a family of (logarithmically) subharmonic functions in a domain $\Omega
 \subseteq \mathbf R^n$. Then
$$u(x):=\sup_{\alpha\in A}u_{\alpha}(x),\quad x \in \Omega$$
is (logarithmically) subharmonic if it is upper semi-continuous in the domain $\Omega$.
\end{proposition}

Also, we need the next theorem due to Vitali (see~\cite{halmos}).

\begin{theorem}[Vitali] Let $X$ be a measurable space with finite measure $\mu$, and
let $h_n:X \to \mathbf C$ be a sequence of functions that are uniformly integrable, i.e.,
such that for every $\epsilon>0$ there exists $\delta>0$, independent of $n$, satisfying
$$\mu(E)<\delta \Rightarrow \int_E |h_n|d\mu < \epsilon.$$
Then if $h_n(x) \to h(x)$ a.e., then
$$\lim_{n \to \infty}\int_X |h_n| d\mu=\int_X |h| d\mu.$$

In particular, if $\sup_n \int_X |h_n| d\mu < \infty$, then the previous condition
holds.
\end{theorem}

\begin{lemma}\label{hlemma} Let $f\in H^p(\mathbf U^2)$. Then
$$\phi(z):=\int_0^{2\pi}|f(z,e^{i\eta})|^p\frac{d\eta}{2\pi}$$
is continuous. Moreover, $\phi$ is logarithmically subharmonic and belongs to the space
$h^1_{PL}$.
\end{lemma}

\begin{proof} For $0\le r<1$, let us denote
$$\phi_r(z):=\int_0^{2\pi}|f(z,re^{i\eta})|^p\frac{d\eta}{2\pi},
\quad z \in \mathbf U.$$
According to Proposition~\ref{ron1}, $\phi_r$ is logarithmically subharmonic in the
unit disc, since $z\to|f(z,re^{i\eta})|^p$ are logarithmically subharmonic for $\eta
 \in [0,2\pi]$. Since for all $z \in \mathbf U$ we have $\phi_r(z) \to \phi (z)$,
monotone as $r\to 1$, it follows that $\phi(z)=\sup_{0\le r<1}\phi_r(z)$. Thus, we
have only to prove (by Proposition~\ref{ron2}) that $\phi$ is continuous.

First of all we have
\begin{equation}\label{hardy}
\phi(z)=\|f(z,\cdot)\|_p \le C_p (1-|z|)^{-\frac 1p}\|f\|_p,\quad z \in \mathbf U,
\end{equation}
for some positive constant $C_p$. Namely, according to the theorem of Hardy and
Littlewood (see~\cite{hardy}, Theorem 27 or~\cite{duren}, Theorem 5.9) applied to
the one variable function $f(\cdot,w)$ with $w$ fixed, we obtain
$$|f(z,w)|\le C_p(1-|z|)^{-\frac 1p}\|f(\cdot,w)\|_p,\quad (z,w)\in\mathbf U^2,$$
for some $C_p>0$. Using the above inequality and the monotone convergence theorem,
we derive
\[\begin{split}
\|f(z,\cdot)\|_p^p &=\lim_{s\to 1}\int_0^{2\pi}|f(z, se^{i\eta})|^p\frac{d\eta}{2\pi}
\\&\le C_p^p(1-|z|)^{-1}\lim_{s\to 1}\int_0^{2\pi}\|f(\cdot,se^{i\eta})\|_p^p\frac{d\eta}{2\pi}
\\&=C_p^p(1-|z|)^{-1}\lim_{s\to 1}\int_0^{2\pi}\frac{d\theta}{2\pi}\lim_{r\to 1} \int_0^{2\pi}|f(re^{i\theta}, se^{i\eta})|_p^p\frac{d\eta}{2\pi}
\\&=C_p^p(1-|z|)^{-1}\lim_{(r,s)\to (1,1)}\int_0^{2\pi}\int_0^{2\pi}|f(re^{i\theta},se^{i\eta})|_p^p\frac{d\theta}{2\pi}\frac{d\eta}{2\pi}
\\&=C_p^p(1-|z|)^{-1}\|f\|_p^p,
\end{split}\]
and~\eqref{hardy} follows.

The inequality~\eqref{hardy} implies that the family of integrals
$$\left\{ \phi(z)=\int_0^{2\pi}|f(z,e^{i\eta})|^p \frac{d\eta}{2\pi}: z \in \mathbf U \right\}$$
is uniformly bounded on compact subsets of the unit disc. Since $z \to |f(z,e^{i\eta})|^p$
is continuous for almost all $\eta \in [0,2\pi]$, as a module of a holomorphic function
(according to~\cite{zygmund}, Theorem XVII 5.16) it follows that $\phi(z),\ z \in \mathbf U$
is continuous. Indeed, let $z_0 \in \mathbf U$ and let $(z_k)_{k\ge1}$ be a sequence
in the unit disc such that $z_k\to z_0,\ k\to\infty$. According to the Vitali theorem
we have
$$\lim_{k\to\infty}\phi(z_k)
=\lim_{k\to\infty}\int_0^{2\pi}|f(z_k,e^{i\eta})|^p\frac{d\eta}{2\pi}
=\int_0^{2\pi}|f(z_0, e^{i\eta})|^p\frac{d\eta}{2\pi}=\phi(z_0).$$
\end{proof}

We now prove the main Theorem~\ref{mainth}.

\begin{proof} Let $f_j \in H^{p_j}(\mathbf U^2),\ j=1,\dots,m$ be holomorphic
functions in the polydisc $\mathbf U^2$. Using the Fubini theorem, Proposition~\ref{burbeadisc},
and Lemma~\ref{loglema}, we obtain
\[\begin{split}
\int_{\mathbf U^2}\prod_{j=1}^m|f_j|^{p_j}d\mu_{(m-2,m-2)}
&=\int_\mathbf U d\mu_{m-2}(z)\int_\mathbf U\prod_{j=1}^m|f_j(z,w)|^{p_j}d\mu_{m-2}(w)
\\&\le\int_\mathbf U d\mu_{m-2}(z)\prod_{j=1}^m\int_0^{2\pi}|f_j(z,e^{i\eta})|^{p_j}\frac{d\eta}{2\pi}
\\&\le\prod_{j=1}^m\int_0^{2\pi}\frac{d\theta}{2\pi}\int_0^{2\pi}|f_j(e^{i\theta},e^{i\eta})|^{p_j}\frac{d\eta}{2\pi}
\\&=\prod_{j=1}^m\int_{\mathbf T^2}|f_j|^{p_j}dm_2,
\end{split}\]
since the functions $\phi_j(z):=\int_0^{2\pi}|f_j(z,e^{i\eta})|^{p_j}\frac{d\eta}{2\pi}$
are logarithmically subharmonic in the disc $\mathbf U$ and since $\phi_j\in h^1_{PL},\
j=1,\dots,m$, by Lemma~\ref{hlemma}.

We now determine when the equalities hold in the above inequalities. Obviously,
if some of functions $f_j,\ j=1,\dots,m$ are identically equal to zero, we have
equalities everywhere. Suppose this is not the case. We will first prove that
$f_j,\ j=1\dots,m$ do not vanish in the polydisc $\mathbf U^2$.

Since for $j=1,\dots,m$ we have $\phi_j \not \equiv 0$, the equality obtains in
the second inequality if and only if for some point $\zeta'' \in \mathbf U$ and
$\lambda_j>0$ we have $\phi_j=\lambda_j|K_1(\cdot,\zeta'')|^2,\ j=1,\dots,m$.
Thus, $\phi_j$ is free of zeroes in the unit disc. Let
$$\psi(z):=\int_\mathbf U \prod_{j=1}^m|f_j(z,w)|^{p_j}d\mu_{m-2}(w),
 \quad z \in \mathbf U.$$
The function $\psi$ is continuous; we can prove the continuity of $\psi$ in a
similar fashion as for $\phi_j$, observing that $\psi(z),\ z \in \mathbf U$ is
uniformly bounded on compact subsets of the unit disc, which follows from the
inequality $\psi(z) \le \prod_{j=1}^m\phi_j(z)$ since the $\phi_j,\ j=1,\dots,m$
satisfy this property. Because of continuity, the equality in the first inequality,
that is,
$$\int_\mathbf U \psi(z)d\mu_{m-2}(z) \le \int_\mathbf U\prod_{j=1}^m\phi_j(z) d\mu_{m-2}(z),$$
holds (by Proposition~\ref{burbeadisc}) only if for all $z \in \mathbf U$ and
some $\zeta'(z) \in \mathbf U$ and $C_j(z) \ne 0$,
$$f_j(z,\cdot)=C_j (z)K_1^\frac 2{p_j}(\cdot,\zeta'(z)),\quad j=1,\dots,m.$$
Since $\phi(z) \ne 0,\ z \in \mathbf U$, it is not possible that $f_j(z,\cdot) \equiv 0$
for some $j$ and $z$.

Thus, if equality holds in~\eqref{maineq}, then $f_j$ does not vanish, $f_j(z,w)
\ne 0,\ (z,w) \in \mathbf U^2$, and we can obtain some branches $f_j^\frac{p_j}2$.
Applying Proposition~\ref{burbeapoly} for $f_j^\frac{p_j}2,\ j=1,\dots,m$, we
conclude that there must hold
$$f_j^\frac{p_j}2=C'_j K_2(\cdot,\zeta),\quad j=1,\dots,m$$
for some point $\zeta \in \mathbf U^2$ and constants $C'_j\ne 0$. The equality
statement of Theorem~\ref{mainth} follows.
\end{proof}

\begin{remark} The generalized polydisc is a product $\Omega^n=\prod_{k=1}^n \Omega_k
\subset \mathbf C^n$, where $\Omega_k,\ k=1,\dots,n$ are simply connected
domains in the complex plane with rectifiable boundaries. Let
$\partial \Omega^n:=\prod_{k=1}^n \partial \Omega_k$ be its distinguished
boundary and let $\phi_k: \Omega_k \to \mathbf U,\ k=1,\dots,n$ be conformal
mappings. Then
$$\Phi(z):=(\phi_1(z_1),\dots,\phi_n (z_n)),\quad z=(z_1,\dots,z_n)$$
is a bi-holomorphic mapping of $\Omega^n$ onto $\mathbf U^n$.

There are two standard generalizations of Hardy spaces on a hyperbolic simple
connected plain domain $\Omega$. One is immediate, by using harmonic majorants,
denoted by $H^p(\Omega)$. The second is due to Smirnov, usually denoted by
$E^p(\Omega)$. The definitions can be found in the tenth chapter of the book of Duren~\cite{duren}.
These generalizations coincide if and only if the conformal mapping of $\Omega$
onto the unit disc is a bi-Lipschitz mapping (by~\cite[Theorem~10.2]{duren});
for example this occurs if the boundary is $C^1$ with Dini-continuous normal
(Warschawski's theorem, see~\cite{war}). The previous can be adapted for generalized
polydiscs (see the paper of Kalaj~\cite{kalaj}). In particular, $H^p(\Omega^n)= E^p(\Omega^n)$
and $\|\cdot\|_{H^p}=\|\cdot\|_{E^p}$, if the distinguished boundary $\partial \Omega^n$
is sufficiently smooth, which means $\partial \Omega_k,\ k=1,\dots,n$ are sufficiently
smooth. Thus, in the case of sufficiently smooth boundary, we may write
$$\|f\|_{H^p(\Omega^n)}
=\left(\frac 1{(2\pi)^n}\int_{\partial \Omega^n}|f(z)|^p|dz_1|\dots|dz_n|\right)
^\frac 1p,$$
where the integration is carried over the non-tangential (distinguished) boundary
values of $f \in H^p(\Omega^n)$.

By Bremerman's theorem (see~\cite[Theorem~4.8, pp. 91--93]{fuks}, $E^2(\Omega^n)$
is a Hilbert space with the reproducing kernel given by
\begin{equation}\label{repkers}
K_{\Omega^n}(z,\zeta):=
K_n(\Phi(z),\Phi(\zeta))\left(\prod_{k=1}^n \phi_k'(z_k)\overline{\phi_k'(\zeta_k)}
\right)^\frac 12,\quad z,\ \zeta \in \Omega^n
\end{equation}
where $K_n$ is the reproducing kernel for $H^2(\mathbf U^n)$; $K_{\Omega^n}$ does not
depend on the particular $\Phi$.

For the next theorem we need the following assertion. The sum $\varphi_1+\varphi_2$
is a logarithmically subharmonic function in $\Omega$ provided $\varphi_1$ and $\varphi_2$
are logarithmically subharmonic in $\Omega$ (see e.g.~\cite[Corollary~1.6.8]{hormander},
or just apply Proposition~\ref{ron1} for a discrete measure $\mu$). By applying this
assertion to the logarithmically subharmonic functions $\varphi_k(z)=|f_k(z)|^2,\ z \in
 \Omega,\ k=1,\dots,l$, where $f=(f_1,\dots,f_l)$ is a $\mathbb C^l$-valued holomorphic
function, and the principle of mathematical induction, we obtain that the function
$\varphi$ defined by
$$\varphi(z):=\|f(z)\|=\left(\sum_{k=1}^l|f_k(z)|^2\right)^\frac 12,\quad z \in \Omega$$
is logarithmically subharmonic in $\Omega$ (obviously, the positive exponent of a
logarithmically subharmonic function is also logarithmically subharmonic).

Theorem~\ref{mainth} in combination with the same approach as in~\cite{burbea}
and~\cite{kalaj} leads to the following sharp inequality for vector-valued
holomorphic functions which generalizes Theorem 3.5 in~\cite[p. 256]{burbea}; by
vector-valued we mean $\mathbb C^l$-valued for some integer $l$. We allow
vector-valued holomorphic functions to belong to the spaces $H^p(\Omega^n)$ if
they satisfy the growth condition~\eqref{growthhardy} with $\|\cdot\|$ instead of
$|\cdot|$.

Let $V_n$ be the volume measure in the space $\mathbf C^n$ and
$$\lambda_{\Omega^n}(z)=K_n(\Phi(z),\Phi(z))\prod_{k=1}^n|\phi'_k(z_k)|,\quad z \in
 \Omega^n$$ be the Poincar\'e metric on the generalized polydisc $\Omega^n$ (the right
side does not depend on the mapping $\Phi$).

\begin{theorem} Let $f_j \in H^{p_j}(\Omega^n),\ j=1,\dots,m$ be holomorphic vector-valued
functions on a generalized polydisc $\Omega^n$ with sufficiently smooth boundary. The
next isoperimetric inequality holds:
$$\frac{(m-1)^n}{\pi^n}
\int_{\Omega^n} \prod_{j=1}^m \|f_j(z)\|^{p_j}\lambda^{2-m}_{\Omega^n}(z)dV_n(z)
\le\prod_{j=1}^m \|f_j\|^{p_j}_{H^{p_j}(\Omega^n)}.$$

For complex-valued functions, the equality in the above inequality occurs if and only
if either some of the $f_j,\ j=1,\dots,m$ are identically equal to zero or if for some
point $\zeta \in \Omega^n$ and constants $C_j\ne0$, or $C_j'\ne0$, the functions have
the following form
$$f_j=
C_j K^\frac 2{p_j}_{\Omega^n}(\cdot,\zeta)=C'_j\left (\prod_{k=1}^n\psi'_k\right)^
\frac 1{p_j}, \quad j=1,\dots,m,$$
where $K_{\Omega^n}$ is the reproducing kernel for the domain $\Omega^n$ and $\psi_k:
\Omega_k \to \mathbf U,\ k=1,\dots,n$ are conformal mappings.
\end{theorem}

In particular, for $n=1$ and $m=2$ and in the case of complex-valued functions, the
above inequality reduces to the result of Mateljevi\'c and Pavlovi\'c mentioned in the
Introduction.
\end{remark}
\subsection*{Acknowledgments}
I wish to thank Professors Miodrag Mateljvi\'c and David Kalaj for very useful comments.
Also, I am grateful to Professor Darko Mitrovi\'c for comments on the exposition and
language corrections.

\end{document}